\documentclass[a4paper,12pt,leqno]{amsart}
\usepackage{amsmath,amsthm,amssymb,latexsym,epsfig,graphicx,subfigure}
\numberwithin{equation}{section}

\newtheorem{thm}{Theorem}[section]

\newtheorem{lm}[thm]{Lemma}

\newtheorem{pr}[thm]{Proposition}
\theoremstyle{definition}
\newtheorem{df}[thm]{Definition}
\theoremstyle{definition}

\newtheorem{conjecture}[thm]{Conjecture}

\newcommand{\Sn}{S^{n-1}}
\newcommand{\Rn}{\mathbb{R}^{n}}

\newcommand{\R}{\mathbb{R}}
\newcommand{\Q}{\mathbb{Q}}
\newcommand{\Z}{\mathbb{Z}}

\newcommand{\M}{\mathcal{M}}

\newcommand {\grtrsim} {\ {\raise-.5ex\hbox{$\buildrel>\over\sim$}}\ }

\newcommand{\khii}{\text{\lower -.4ex\hbox{$\chi$}}}

\DeclareMathOperator{\Int}{Int}

\begin{document}
\title{Hausdorff dimension, projections, intersections, and Besicovitch sets}
\author{Pertti Mattila}

\thanks{The author was supported by the Academy of Finland through the Finnish Center of Excellence in Analysis and Dynamics Research} 

\begin{abstract} 
This is a survey on recent developments on the Hausdorff dimension of projections and intersections for general subsets of Euclidean spaces, with an emphasis on estimates of the Hausdorff dimension of exceptional sets and on restricted projection families. We shall also discuss relations between projections and Hausdorff dimension of Besicovitch sets.

\end{abstract}

\maketitle

\section{Introduction}

In this survey I shall discuss some recent results on integral-geometric properties of Hausdorff dimension and their relations to Kakeya type problems. More precisely, by integral-geometric properties I mean properties related to affine subspaces of Euclidean spaces and to rigid motions; orthogonal projections into planes, intersections with planes, and intersections of two sets after a generic rigid motion is applied to one of them. Such questions have been studied for more than 60 years and there have been a lot of recent activities on them. In particular, I shall discuss estimates on the Hausdorff dimension of exceptional sets of planes and rigid motions, and projections on restricted families of planes. Besicovitch sets are sets of Lebesgue measure zero containing a unit line segment in every direction. They are expected to have full Hausdorff dimension. This problem is related to many topics in modern Fourier analysis. It is also related to projection theorems, as we shall see at the end of this survey. In the last section I shall also discuss $(n,k)$ Besicovitch sets, lines replaced with $k$-planes, and their relations to projections.

Other recent surveys partially overlapping with this are \cite{FFJ}, \cite{Ke3}, \cite{S} and \cite{M5}.

Most of the background material can be found, for example, in the books \cite{M4} and \cite{M6}.

This survey is partially based on the lectures I gave in the CIMPA2017 conference in Buenos Aires in August 2017. I would like to thank Ursula Molter, Carlos Cabrelli and the other organizers for that very pleasant and succesful event. I am grateful to Tuomas Orponen for many useful comments.

\section{Hausdorff dimension, energy integrals and the Fourier transform}
I give here a quick review of the Hausdorff dimension and its relations to energy-integrals and the Fourier transform. The details can be found in \cite{M4} and \cite{M6}.

The $s$-dimensional \emph{Hausdorff measure} 
$\mathcal H^s, s\geq 0$, is defined for $A\subset\R^n$ by
$$\mathcal H^s(A)=\lim_{\delta\to0} \mathcal H_{\delta}^s(A),$$
where, for $0<\delta\leq\infty$, 
$$\mathcal H_{\delta}^s(A)=\inf\{\sum_{j=1}^{\infty}d(E_j)^s:A\subset\bigcup_{j=1}^{\infty}E_j, d(E_j)<\delta\}.$$
Here $d(E)$ denotes the diameter of the set $E$.  

Then $\mathcal H^n$ is a constant multiple of the Lebesgue measure $\mathcal L^n$ and the restriction of $\mathcal H^{n-1}$ to the unit sphere $\Sn=\{x\in\Rn:|x|=1\}$ is a constant multiple of the surface measure.

The \emph{Hausdorff dimension} 
of  $A$  is
$$\dim A=\inf\{s:\mathcal H^s(A)=0\}=\sup\{s:\mathcal H^s(A)=\infty\}.$$

For $A\subset\R^n$, let $\mathcal M(A)$ be the set of Borel measures $\mu$ such that $0<\mu(A)<\infty$ and $\mu$ has compact support spt$\mu\subset A$. We denote by $B(x,r)$ the closed ball with center $x$ and radius $r$. The following is a useful tool for proving lower bounds for the Hausdorff dimension:

\begin{thm}[Frostman's lemma]
Let $0\leq s\leq n$. For a Borel set $A\subset\R^n, \mathcal H^s(A)>0$ if and only there is $\mu\in\mathcal M(A)$ such that 
\begin{equation}\label{frostman1}
\mu(B(x,r))\leq r^s\quad \text{for all}\ x\in\R^n, r>0.
\end{equation}
In particular,
$$\dim A=\sup\{s:\text{there is}\ \mu\in\mathcal M(A)\ \text{such that (\ref{frostman1}) holds}\}.$$
\end{thm}

Such measures $\mu$ are often called Frostman measures.

The  $s$-\emph{energy}, $s>0$, of a Borel measure $\mu$ is
$$I_s(\mu)=\iint|x-y|^{-s}\,d\mu x\,d\mu y=\int k_s\ast\mu\, d\mu,$$
where $k_s$ is the \emph{Riesz kernel}:
$$k_s(x)=|x|^{-s},\quad x\in\Rn.$$

Integration of Frostman's lemma gives

\begin{thm}\label{energy} 
For a Borel set $A\subset\Rn,$
$$\dim A=\sup\{s:\text{there is}\ \mu\in\mathcal M(A)\ \text{such that}\ I_s(\mu)<\infty\}.$$
\end{thm}

The \emph{Fourier transform} of $\mu\in\mathcal M(\Rn)$ is
$$\widehat{\mu}(\xi)=\int e^{-2\pi i\xi\cdot x}\,d\mu x,\quad \xi\in\Rn.$$
The $s$-energy of $\mu\in\mathcal M(\Rn)$ can be written in terms of the Fourier transform:
$$I_s(\mu)=c(n,s)\int |\widehat{\mu}(x)|^2|x|^{s-n}\,dx.$$
This comes from Plancherel's theorem and the fact that the Fourier transform, in the distributional sense, of $k_s$ is a constant multiple of $k_{n-s}$. Thus we have 
\begin{equation}\label{dim}\dim A=
\sup\{s<n:\exists \mu\in\mathcal M(A)\ \text{such that}\ \int |\widehat{\mu}(x)|^2|x|^{s-n}\,dx<\infty\}.
\end{equation}

Notice that if $I_s(\mu) < \infty$, then  $|\widehat{\mu}(x)|^2 < |x|^{-s}$ for most $x$ with large norm. However, this need not hold for all $x$ with large norm.

The upper Minkowski dimension is defined by
$$\dim_MA=\inf\{s\geq 0:\lim_{\delta\to 0}\delta^{s-n}\mathcal L^n(\{x: dist(x,A)<\delta\})=0\}.$$
The packing dimension $\dim_P$ can be defined as a modification of this:
$$\dim_PA=\inf\{\sup_i\dim_MA_i:A=\bigcup_{i=1}^{\infty}A_i\}.$$
Then $\dim A\leq\dim_PA\leq\dim_MA$. We have the following product inequalities:
\begin{equation}\label{producth}
\dim A\times B \geq \dim A + \dim B.
\end{equation}
\begin{equation}\label{productm}
\dim_M A\times B \leq \dim_M A + \dim_M B.
\end{equation}
\begin{equation}\label{productp}
\dim_P A\times B \leq \dim_P A + \dim_P B.
\end{equation}

There is no Fubini theorem for Hausdorff measures, but we have the following inequality, see \cite{Fe}, 2.10.25:
\begin{pr}\label{sectprop} Let $A\subset\R^{m+n}$ and set $A_x=\{y\in\Rn: (x,y)\in A\}$ for $x\in\R^m$. Then for any non-negative numbers $s$ and $t$ ($\int^{\ast}$ is the upper integral)
$$\int^{\ast}\mathcal H^t(A_x)\,d\mathcal H^sx \leq C(m,n,s,t)\mathcal H^{s+t}(A).$$
In particular, if $\dim\{x\in\R^m: \dim A_x\geq t\}\geq s$, then $\dim A \geq s+t$.
\end{pr}

The latter statement was proved by Marstrand in \cite{Ma2}.

\section{Hausdorff dimension and exceptional projections}

We shall now discuss the question: how do orthogonal projections 
affect the Hausdorff dimension? Let $0<m<n$ be integers and let $G(n,m)$ be the space of all linear $m$-dimensional subspaces of $\Rn$ and let $\gamma_{n,m}$ be the Borel probability measure on it which is invariant under the orthogonal group $O(n)$ of $\Rn$. For $V\in G(n,m)$ let $P_V:\Rn\to V$ be the orthogonal projection. 

The case $m=1$ and the lines through the origin is simpler and more concrete, and perhaps good to keep in mind. We can parametrize $G(n,1)$ and the projections onto lines by the unit sphere:
$$P_e(x) = e\cdot x,\quad x\in\Rn, e\in\Sn.$$ 
 
Here is the basic projection theorem for the Hausdorff dimension. 
The first two items of it were proved by Marstrand \cite{Ma1} in 1954 and the third by Falconer and O'Neil \cite{FO} in 1999 and by Peres and Schlag \cite{PS} in 2000:

\begin{thm}\label{projections}

Let $A\subset\R^n$ be a Borel set. 
\begin{itemize}
\item[(1)] If $\dim A \leq m$, then
\begin{equation*} \dim P_V(A) = \dim A\quad \text{for}\ \gamma_{n,m}\ \text{almost all}\ V \in G(n,m). 
\end{equation*}
\item[(2)] If $\dim A > m$, then
\begin{equation*}  \mathcal L^{m}(P_V(A)) > 0\quad \text{for}\ \gamma_{n,m}\ \text{almost all}\ V \in G(n,m).  \end{equation*}
\item[(3)] If $\dim A > 2m$, then
$P_V(A)$ has non-empty interior for $\gamma_{n,m}$ almost all $V \in G(n,m).$
\end{itemize}
\end{thm}

\begin{proof}
We only prove this for $m=1$, the general case can be found in \cite{M6}. 
For $\mu\in\M(A)$, let $\mu_e\in\M(P_{e}(A))$ be the push-forward of $\mu$ under $P_{e}$: $\mu_e(B)=\mu(P_{e}^{-1}(B))$.

To prove (1) let $0<s<\dim A$ and choose by Theorem \ref{energy} a measure $\mu\in\M(A)$ such that $I_s(\mu)<\infty$. Then 
\begin{align*}
\int_{S^{n-1}}I_s(\mu_e)\,de&=\int_{S^{n-1}}\iint |P_e(x-y)|^{-s}\,d\mu x\,d\mu y\,de\\
&=\iiint_{S^{n-1}}|P_e(\tfrac{x-y}{|x-y|})|^{-s}\,de|x-y|^{-s}\,d\mu x\,d\mu y=c(s)I_s(\mu)<\infty,
\end{align*}
where for $v\in S^{n-1}$, $c(s)=\int_{S^{n-1}}|P_e(v)|^{-s}\,de<\infty$ as $s<1$. The finiteness of this integral follows from the simple inequality
\begin{equation}\label{projineq}
\mathcal H^{n-1}(\{e\in \Sn: |P_e(x)|\leq\delta\})\lesssim \delta/|x|\quad \text{for}\ x\in\Rn\setminus\{0\}, \delta>0.
\end{equation}
Referring again to Theorem \ref{energy} we see that $\dim P_{e}(A)\geq s$ for almost all $e\in S^{n-1}$. By the arbitrariness of 
$s, 0<s<\dim A$, we obtain $\dim P_{e}(A)\geq\dim A$ for almost all $e\in S^{n-1}$. The opposite inequality follows from the fact that the projections are Lipschitz mappings. 

To prove (2) choose by (\ref{dim}) a measure $\mu\in\M(A)$ such that $\int|x|^{1-n}|\widehat{\mu}(x)|^2\,dx<\infty.$
Directly from the definition of the Fourier transform we see that $\widehat{\mu_e}(t)=\widehat{\mu}(te)$ for $t\in\R, e\in S^{n-1}$. Integrating in polar coordinates we obtain
$$\int_{S^{n-1}}\int_{-\infty}^{\infty}|\widehat{\mu_e}(t)|^2\,dt\,de=2\int_{S^{n-1}}\int_0^{\infty}|\widehat{\mu}(te)|^2\,dt\,de=2\int|x|^{1-n}|\widehat{\mu}(x)|^2\,dx<\infty.$$
Thus for almost all $e\in S^{n-1}$, $\widehat{\mu_e}\in L^2(\R)$ which means that $\mu_e$ is absolutely continuous with $L^2$ density and hence $\mathcal L^1(p_{e}(A))>0$.

For the proof of (3) one takes  $2 < s < \dim A$ and $\mu\in\M(A)$ such that $I_s(\mu)<\infty$, whence $\int|x|^{s-n}|\widehat{\mu}(x)|^2\,dx<\infty.$ Then as above and by the Schwartz inequality
\begin{align*}
\int_{S^{n-1}}\int_{|t|\geq 1}|\widehat{\mu_e}(t)|\,dt\,de = 2\int_{|x|\geq 1}|x|^{1-n}|\widehat{\mu}(x)|\,dx\\
\leq 2\left(\int_{|x|\geq 1}|x|^{2-s-n}\,dx\int_{|x|\geq 1}|x|^{s-n}|\widehat{\mu}(x)|^2\,dx\right)^{1/2}<\infty
\end{align*}
since $2-s-n < -n$. Thus for almost all $e\in S^{n-1}$, $\widehat{\mu_e}\in L^1(\R)$ which implies that $\mu_e$ is absolutely continuous with continuous density. Hence $P_e(A)$ has non-empty interior.
\end{proof}

Part (2) can rather easily be proven also without the Fourier transform using again inequalities like \eqref{projineq}, see the proof of  Theorem 9.7 in \cite{M4}. Parts (1) and (2) of Theorem \ref{projections} hold with $\gamma_{n,m}$ replaced with any Borel   measure $\gamma$ on $G(n,m)$ which satisfies
$$\gamma(\{V\in G(n,m): |P_V(x)|\leq\delta\})\lesssim (\delta/|x|)^m\quad \text{for}\ x\in\Rn\setminus\{0\}, \delta>0.$$
We shall discuss this a bit more later. I don't know any proof for (3) without the Fourier transform.

The conditions $\dim A \leq m$ and $\dim A > m$ in (1) and (2) are of course necessary. The condition $\dim A > 2m$ in (3) is necessary if $m=1$. I don't know if it is necessary when $m>1$. In the case $m=1$ the example in the plane can be obtained with Besicovitch sets, first in the plane, showing that there is no theorem in the plane, and then taking cartesian products. More precisely, let $B\subset \R^2$ be a Borel set of measure zero which contains a line in every direction. We shall construct such sets in Section 7. Let $A=\R^2\setminus\cup_{q\in\Q^2}(B+q)$, where $\Q^2$ is the countable dense set with rational coordinates. Then $A$ has full Lebesgue measure and none of its projections has interior points.

In this section we shall discuss how much more one can say about the size of the sets of exceptional planes. Kaufman \cite{Ka} proved in 1968 the first item of the following theorem in the plane (generalized in \cite{M1}), Falconer \cite{F2} in 1982 the second and Peres and Schlag \cite{PS} in 2000 the third. Recall that the dimension of $G(n,m)$ is $m(n-m)$. To get a better feeling of this notice that in the case $m=1$ the three upper bounds are $n-2+\dim A, n-\dim A$ and $n+1-\dim A$.

\begin{thm}\label{excproj}

Let $A\subset\R^n$ be a Borel set. 
\begin{itemize}
\item[(1)] If $\dim A \leq m$, then
\begin{equation*} \dim\{V\in G(n,m): \dim P_{V}(A) < \dim A\} \leq m(n-m) - m + \dim A. 
\end{equation*}
\item[(2)] If $\dim A > m$, then
\begin{equation*}  \dim\{V\in G(n,m): \mathcal L^{m}(P_{V}(A))=0\} \leq m(n-m)+m-\dim A.  \end{equation*}
\item[(3)] If $\dim A > 2m$, then 
\begin{equation*}
\dim\{V\in G(n,m): \Int(P_{V}(A))=\emptyset\} \leq m(n-m)+2m-\dim A.  \end{equation*}
\end{itemize}
\end{thm}

The proof of (1) is a rather simple modification of the proof of the corresponding part in Theorem \ref{projections}; essentially one just replaces the measure $\gamma_{n.m}$ with a Frostman measure $\nu$ on the exceptional set. The key observation is that instead of \eqref{projineq} we now have
\begin{equation}\label{projineq1}
\nu(\{V\in G(n,m): |P_V(x)|\leq\delta\})\lesssim (\delta/|x|)^{s-m(n-m-1)}
\end{equation}
which easily follows from the Frostman condition $\nu(B(V,r))\leq r^s$, cf. \cite{M6}, (5.10) and (5.12). 
 The proofs of (2) and (3) are trickier and require the use of the Fourier transfom. They can be found in \cite{M6}. 
 
Theorem \ref{excproj} and much more, for instance exceptional set estimates for Bernoulli convolutions, is included in the setting of generalized projections developed by Peres and Schlag in \cite{PS}. Later these general estimates have been improved in many special cases. 

The bounds in (1) and (2) are sharp by the examples which Kaufman and I constructed in 1975 in \cite{KM}. I don't know if the bound in (3) is sharp. Another, seemingly very difficult, problem is estimating the dimension of the set in (1) when $\dim A$ is replaced by some $u < \dim A$. We still have by the same proof
\begin{equation*} \dim\{V\in G(n,m): \dim P_{V}(A) < u\} \leq m(n-m) - m + u, 
\end{equation*}
but this probably is not sharp when $u<\dim A$. In any case it is far from sharp in the plane when $u=\dim A/2$:

\begin{thm}\label{bour}
Let $A\subset\R^2$ be a Borel set. Then
\begin{equation} \dim\{e\in S^1: \dim P_{e}(A) \leq \dim A/2\} = 0. 
\end{equation}
\end{thm}

To get some idea where $\dim A/2$ comes from, notice that the inequality $\dim_M P_{e}(A) < \dim_M A/2$ is very easy for the upper Minkowski dimension 
(and also for the packing dimension), and even more is true: there can be at most one direction $e$ for which  $\dim_M P_{e}(A) < \dim_M A/2$. That there cannot be two orthogonal directions follows immediately from the product inequalities \eqref{productm} and \eqref{productp}, and the general case is also easy. However for the Hausdorff dimension the exceptional set can always be uncountable, even more: Orponen constructed in \cite{O3}, Theorem 1.5, a compact set $A\subset\R^2$ such that $\mathcal H^1(A)>0$ and $\dim\{e\in S^1: \dim P_e(A) =0\}$ is a dense $G_{\delta}$ subset of $S^1$. This paper also contains many exceptional set estimates for projections and packing dimension.

Theorem \ref{bour} is due to Bourgain, \cite{B3}, \cite{B4}. Bourgain's result is more general and it includes a deep discretized version. The proof uses methods of additive combinatorics. D. M. Oberlin gave a simpler Fourier-analytic proof in \cite{Ob2}, but with $\dim P_{e}(A) \leq \dim A/2$ replaced by $\dim P_{e}(A) < \dim A/2$. Using combinatorial methods He \cite{H} proved analogous higher dimensional results. 

More generally, it might be true, and has been conjectured by Oberlin \cite{Ob2}, that Kaufman's estimate

\begin{equation}\label{kauf}  \dim\{e\in S^1: \dim P_e(A) < u\} \leq u  
\end{equation}
could be extended for $\dim A/2\leq u \leq \dim A$ to
\begin{equation}\label{Ob}  \dim\{e\in S^1: \dim P_e(A) < u\} \leq 2u -\dim A.  
\end{equation}

This would be sharp, as the constructions in \cite{KM} show. Theorem \ref{bour} is the only case where this is known. However, Orponen improved in the plane  Theorem in \ref{excproj} in \cite{O6} and \cite{O7} for sets $A$ with $\dim A=1$ but with $\dim P_e(A)$ replaced by the packing dimension of $P_e(A)$: for $0<t<1$ there is $\epsilon(t)>0$ such that
\begin{equation}\label{O2}
\dim\{e\in S^1: \dim_P P_e(A) < t\} \leq t -\epsilon(t).
\end{equation}

The following generalization of parts (1) and (2) of Theorem \ref{projections} tells us that that a null set of projections can be found first and then the statements hold outside these exceptions for all subsets of positive measure. Statement (2) is due to Marstrand \cite{Ma1}. It means that the pushforward under $P_V$ of the restriction of $\mathcal H^s$ to $A$ is absolutely continuous for almost all $V\in G(n,m)$, recall the proof of Theorem \ref{projections}(2). Part (1) was proved by Falconer and the author in \cite{FM}.

\begin{thm}\label{marthm}
Let  $A \subset \mathbb{R}^n$ be an $\mathcal H^s$-measurable set with $0<\mathcal H^s(A) < \infty$. Then there exists a Borel set $E\subset G(n,m)$ with $\gamma_{n,m}(E) = 0$ such that for all $V \in G(n,m)\setminus E$ and all 
$\mathcal H^s$-measurable sets $B\subset A$ with $\mathcal H ^s(B)>0$,

${\rm (1)}$ if $s\leq m$ then $\dim P_V(B) = s$,
\smallskip

${\rm (2)}$  if $s >m$ then $\mathcal L^m (P_V(B)) > 0$.
\smallskip

\end{thm}

The sharper version in the spirit of Theorem \ref{excproj} is also valid, see \cite{FM}.

In the next section the following theorem will give us information about exceptional plane slices. It was proved by Orponen and the author in \cite{MO}:

\begin{thm}\label{main1}
Let $A$ and $B$ be Borel subsets of $\Rn$.
\begin{itemize}
\item[(1)] If $\dim A > m$ and  $\dim B > m$, then 
$$\gamma_{n,m}\left(\{V \in G(n,m): \mathcal L^{m}(P_V(A)\cap P_V(B))>0\}\right)>0.$$
\item[(2)] If $\dim A > 2m$ and  $\dim B > 2m$, then 
$$\gamma_{n,m}\left(\{V \in G(n,m): \Int(P_V(A)\cap P_V(B))\not=\emptyset\}\right)>0.$$
\item[(3)] If $\dim A > m, \dim B \leq m$ and $\dim A + \dim B > 2m$, then for every $\epsilon > 0$,
$$\gamma_{n,m}\left(\{V \in G(n,m): \dim(P_V(A)\cap P_V(B))>\dim B - \epsilon\}\right)>0.$$
\end{itemize}
\end{thm}

\begin{proof}
I only prove (1) when $m=1$. 
Choose by (\ref{dim}) $\mu\in\M(A)$ and $\nu\in\M(B)$ such that $\int|x|^{1-n}|\widehat{\mu}(x)|^2\,dx<\infty$ and $\int|x|^{1-n}|\widehat{\nu}(x)|^2\,dx<\infty.$
Let again $\mu_e\in\M(P_{e}(A))$ and $\nu_e\in\M(P_{e}(B))$ be the push-forwards of $\mu$ and $\nu$ under $P_{e}$. We know from the proof of Theorem \ref{projections} that for almost all $e\in S^{n-1}$, $\mu_e$ and $\nu_e$ are absolutely continuous with $L^2$ densities. Thus as in the proof of Theorem \ref{projections} and by Plancherel's theorem, 
\begin{align*}
\iint &\mu_e(t) \nu_e(t)\, dt\,de  
= \iint \widehat{\mu_e}(t) \overline{\widehat{\nu_e}(t)}\, dt\,de
= \iint \widehat{\mu}(te) \overline{\widehat{\nu}(te)}\, dt\,de\\
&=c(n)\int_{\Rn} |x|^{1-n}\widehat{\mu}(x) \overline{\widehat{\nu}(x)}\, dx = c(n,m)\iint|x-y|^{-1}\,d\mu x\,d\nu x>0.\notag
\end{align*}
Hence $\int \mu_e(t) \nu_e(t)\, dt > 0$ for positively many $e$. As $\mu_e \nu_e$ has support in $P_e(A)\cap P_e(B)$, the claim follows.
\end{proof}

For other recent projection results, see \cite{BI1}, \cite{BI2}, \cite{C1}, \cite{C2} and \cite{BFVZ}.

There are many recent results on projections of various special, for example self-similar, classes of sets and measures. I shall not discuss them here but \cite{FFJ} and \cite{S} give good overviews. 

\section{Restricted families of projections}
Here we discuss the question: what kind of projection theorems can we get if the whole Grassmannian $G(n,m)$ is replaced by some lower dimensional subset $G$? A very simple example is the one where $G\subset G(3,1)$ corresponds to a circle in a two-dimensional plane in $\R^3$. For example, we can consider the projections $\pi_{\theta}$ onto the lines $\{t(\cos\theta,\sin\theta,0): t\in\R\}, \theta\in [0,\pi]$. Since $\pi_{\theta}(A)=\pi_{\theta}((\pi(A))$ where $\pi(x,y,z)=(x,y)$, and $\dim A\leq\dim\pi(A)+1$,  it is easy to conclude using Marstrand's projection Theorem \ref{projections} that for any Borel set $A\subset\R^3$, for almost all $\theta\in [0,\pi]$,
\begin{align*}
&\dim \pi_{\theta}(A)\geq \dim A-1\quad \text{if}\ \dim A\leq 2,\\
&\mathcal L^1(\pi_{\theta}(A))>0\quad \text{if}\ \dim A> 2.
\end{align*}
This is sharp by trivial examples; consider product sets $A=B\times\ C, B\subset\R^2,  C\subset\R$.  So we only have an essentially trivial result. The situation changes dramatically if we consider the projections $p_{\theta}$ onto the lines $\{t(\cos\theta,\sin\theta,1): t\in\R\}$. Then the trivial counter-examples do not work anymore and one can now improve the above estimates. The method used for the proof of Theorem \ref{projections} easily gives that if $A\subset\R^3$ is a Borel set with $\dim A\leq 1/2$, then 
$$\dim p_{\theta}(A)\geq \dim A\quad \text{for almost all}\ \theta\in [0,\pi].$$ 
The restriction $1/2$ comes from the fact that instead of \eqref{projineq} we now have only
\begin{equation}\label{projineq2}
\mathcal L^{1}(\{\theta: |p_{\theta}(x)|\leq\delta\})\lesssim \sqrt{\delta/|x|}. 
\end{equation} 

For $\dim A > 1/2$ this becomes much more difficult. Anyway we have 

\begin{thm}\label{kov}
Let $p_{\theta}$ and $q_{\theta}$ be the orthogonal projections onto the line $\{t(\cos\theta,\sin\theta,1): t\in\R\}, \theta\in[0,\pi]$, and its orthogonal complement. Let $A\subset \R^3$ be a Borel set. 
\begin{itemize}
\item[(1)] If $\dim A\leq 1$, then $\dim p_{\theta}(A)=\dim A$ for almost all $\theta\in[0,\pi]$.
\item[(2)] If $\dim A\leq 3/2$, then $\dim q_{\theta}(A)=\dim A$ for almost all $\theta\in[0,\pi]$.
\end{itemize}
\end{thm}

K\"aenm\"aki, Orponen and Venieri proved (1) in \cite{KOV} and Orponen and Venieri (2) in \cite{OV}. 
They related this problem to circle packing problems and methods of Wolff from \cite{W2}. 

So (1) is the sharp analogue of the corresponding part of Marstrand's projection theorem for these projections. Perhaps (2) is not sharp in the sense that it might hold with $2$ in place of $3/2$. 

One reason for the possibility of such improvements over the first family of projections considered above, the $\pi_{\theta}$, is that the second family, the $p_{\theta}$, is more curved than the first one. That is, the set of the unit vectors generating the first family is the planar curve  $\{(\cos\theta,\sin\theta,0): \theta\in [0,\pi]\}$ while for the second it spans the whole space $\R^3$. More precisely, the curve $\gamma(\theta)=(\cos\theta,\sin\theta,1)/\sqrt{2}\in S^2, \theta\in[0,\pi]$, of the corresponding unit vectors satisfies the curvature condition that for every $\theta\in [0,\pi]$ the vectors $\gamma(\theta), \gamma'(\theta),\gamma''(\theta)$ span the whole space $\R^3$. Partial results were proven earlier by F\"assler and Orponen \cite{FOr}, \cite{O2} and  D. M.  Oberlin and R. Oberlin  \cite{OO} for general $C^2$ curves on $S^2$ satisfying this curvature condition. F\"assler and Orponen conjectured that the  full  Marstrand theorem as in Theorem \ref{kov} (with $3/2$ replaced by $2$) should hold for them.

As we have seen above, if $\rho_{e}:\R^3\to\R, e\in S^2,$ is a family of linear mappings and $\sigma$ is a Borel measure on $S^2$ satisfying
\begin{equation*}
\sigma(\{e: |\rho_{e}(x)|\leq\delta\})\lesssim \delta/|x|, 
\end{equation*} 
then the Marstrand statement $\dim\rho_{\theta}(A)=\min\{\dim A,1\}$ holds for $\sigma$ almost all $e\in S^2$. However such inequality is usually false for less than 2-dimensional measures $\sigma$. Nevertheless Chen constructed in \cite{C2} for all $1<s<2$~  $s$-dimensional Ahlfors-David regular random measures for which it holds, and hence also the Marstrand theorem. He had also many other related results in that paper.

Next we consider projection families in higher dimensions. I state a more general result below but let us start with
$$\pi_t:\R^4\to\R^2, \pi_t(x,y) = x + ty, x,y\in\R^2, t\in\R.$$
This family is closely connected with Besicovitch sets and the Kakeya conjecture in $\R^3$, as we shall later see. The following theorem is due to D. M. Oberlin \cite{Ob3}. It is not explicitly stated there but follows from the proof of Theorem 1.3.

\begin{thm}\label{obproj}
Let $A\subset\R^4$ be a Borel set. 
\begin{itemize}
\item[(1)] If $\dim A \leq 3$, then $\dim\pi_t(A)\geq\dim A - 1$ for almost all $t\in\R$.
\item[(2)]If $\dim A > 3$, then $\mathcal L^2(\pi_t(A))>0$ for almost all $t\in\R$.
\end{itemize}
\end{thm}

The bounds here are sharp when $\dim A\geq 2$. To see this let $0\leq s \leq 1, C_s\subset\R$ with $\dim C_s=s$, and $A_s = \{(x,y)\in\R^2\times\R^2:x_1\in C_s, y_1=0\}.$ Then $\dim A_s = 2+s, \pi_t(A_s) =  C_s\times\R$ and $\dim \pi_t(A_s) = 1+s$. This shows that (1) is sharp. For (2) we can choose $C_1$ with $\mathcal L^1(C_1)=0$, then $\mathcal L^2(\pi_t(A))=0$. These bounds are not sharp for all $A$ since we have $\dim\pi_t(A)=\dim A$ for almost all $t\in\R$ if $\dim A\leq 1$. Restricting $t$ to some interval $[c,C], 0<c<C<\infty,$ this follows as before from the inequality
$$\mathcal L^{1}(\{t\in[c,C]: |\pi_{t}(x,y)|\leq\delta\})\lesssim \delta/|(x,y)|,$$
which is easy to check. 
If $1\leq\dim A\leq 2$ we can only say that $\dim\pi_t(A)\geq 1$ for almost all $t\in\R$ since $\pi_t(\R\times\{0\}\times\R\times\{0\})=\R$. 

I give a sketch of the proof of Theorem \ref{obproj}. Let $\mu\in\mathcal M(A)$ with 
\begin{equation}
\mu(B(x,r)) \leq r^s\quad \text{for}\ x\in\R^4, r>0, 
\end{equation}
for some $0<s<4$.
 Let $\mu_t\in\mathcal M(\pi_t(A))$ be the push-forward of $\mu$ under $\pi_t$. Then for $\xi\in\Rn$,
$$\widehat{\mu_t}(\xi)=\int e^{-2\pi i\xi\cdot\pi_t(x,y)}\,d\mu (x,y)=\int e^{-2\pi i(\xi,t\xi)\cdot(x,y)}\,d\mu (x,y)=
\widehat{\mu}(\xi,t\xi).$$

It is enough to consider $t$ in some fixed bounded interval $J$. Oberlin proved that for $R>0$,
\begin{equation}\label{obeq}
\int_J\int_{R\leq|\xi|\leq 2R}|\widehat{\mu}(\xi,t\xi)|^2\,d\xi\,dt\lesssim R^{4-s-1}.
\end{equation}
This is applied to the dyadic annuli, $R=2^k, k=1,2,\dots$. The sum converges if $s>3$, and we can choose $\mu$ with such $s$ if $\dim A > 3$. This gives $\int_J\int|\widehat{\mu_t}(\xi)|^2\,d\xi\,dt<\infty$ and yields part (2). To prove part (1) let $0<u<s<\dim A$ and $\mu$ as above. Then \eqref{obeq} yields
\begin{equation*}
\int_J\int|\widehat{\mu_t}(\xi)|^2|\xi|^{u-1-2}\,d\xi\,dt<\infty,
\end{equation*}
so $\dim\pi_t(A)\geq u - 1$ for almost all $t\in J$ and thus $\dim\pi_t(A)\geq\dim A - 1$ for almost all $t\in\R$ by the arbitrariness of $J$ and $u$.

Let us formulate \eqref{obeq} as a more general lemma (a special case of Lemma 3.1 in \cite{Ob3}): 

\begin{lm}\label{oblemma}
 Let $k$ and $m$ be positive integers and $N=(k+1)m$ and let $Q$ be a cube in $\R^k$. Define 
\begin{equation*}
T_t\xi=(t_1\xi,\dots,t_k\xi)\in\R^{km}\quad \text{for}\ \xi\in\R^m, t\in\R^k.
\end{equation*}

If $\mu\in\mathcal M(\R^N)$ with $\mu(B(x,r)) \leq r^s$ for $x\in\R^N, r>0$ for some $0<s<n$, then 
\begin{equation}\label{obeq2}
\int_Q\int_{R\leq|\xi|\leq 2R}|\widehat{\mu}(\xi,T_t\xi)|^2\,d\xi\,dt\lesssim R^{N-s-k}.\end{equation}
\end{lm}

We obtain \eqref{obeq} from this with $k=1, m=2$. 

In Lemma 3.1 of \cite{Ob3} there is an additional assumption (3.1). This is now trivial: it is applied with $\lambda$ equal to the Lebesgue measure on $Q$. See the proof of Theorem 1.3 in \cite{Ob3} for the identification of our Lemma \ref{oblemma} as a special case of Lemma 3.1 of \cite{Ob3}.

To prove Lemma \ref{oblemma}, choose a smooth function $g$ with compact support which equals 1 on the support of $\mu$. Then $\widehat{g\mu}=\widehat{g}\ast\widehat{\mu}$ and the integral in \eqref{obeq2} equals
$$\int_Q\int_{R\leq|\xi|\leq 2R}|\widehat{g\mu}(\xi,T_t\xi)|^2\,d\xi\,dt
=\int_Q\int_{R\leq|\xi|\leq 2R}\left|\int\widehat{g}((\xi,T_t\xi)-y)\widehat{\mu}(y)\,dy\right|^2\,d\xi\,dt.$$
This can be estimated by standard arguments. When $|y|$ is large as compared to $R$, $|\widehat{g}((\xi,T_t\xi)-y)|$ is small by the fast decay $\widehat{g}$. For $|y|\lesssim R$ one uses 
$$\int_{|y|\leq CR}|\widehat{\mu}(y)|^2\,dy \lesssim R^{s-N},$$
which follows from the assumption $\mu(B(x,r)) \leq r^s$, cf. also \cite{M6}, Section 3.8. Of course, I am skipping several technical details here, see \cite{Ob3}.

We now formulate a more general version of the above projection theorem. Let $k$ and $m$ be positive integers and $N=(k+1)m$. Above we had $k=1, m=2$. Write
$$x=(x_0^1,\dots,x_0^m,x_1^1,\dots,x_1^m,\dots,x_k^1,\dots,x_k^m)\in\R^N, t=(t_1,\dots,t_k)\in\R^k.$$
Consider the linear mappings
\begin{align*}
\pi_t:\R^N\to\R^m, \pi_t(x)&=(x_0^1+\sum_{j=1}^kt_jx_j^1,\dots,x_0^m+\sum_{j=1}^kt_jx_j^m)\\
&=(x_0^1+t\cdot x^1,\dots,x_0^m+t\cdot x^m)=x_0+t\cdot \tilde{x},
\end{align*}
where $x_0=(x_0^1,\dots,x_0^m), x^l=(x_1^l,\dots,x_k^l)$ and $t\cdot \tilde{x}=(t\cdot x^1,\dots,t\cdot x^m)\in\R^m$. Then for $\mu\in \mathcal M(\R^N)$ the push-forward $\mu_t$ of $\mu$ under $\pi_t$ has the Fourier transform for $\xi\in\R^m$, 
$$\widehat{\mu_t}(\xi)=\int e^{-2\pi i\xi\cdot\pi_t(x)}\,d\mu x=\int e^{-2\pi i(\xi\cdot x_0+\xi\cdot(t\cdot \tilde{x}))}\,d\mu x
=\widehat{\mu}(\xi,T_t\xi),$$
where again $T_t\xi=(t_1\xi,\dots,t_k\xi)\in\R^{km}$. Lemma \ref{oblemma} now yields
\begin{equation}\label{obeq1}
\int_Q\int_{R\leq|\xi|\leq 2R}|\widehat{\mu_t}(\xi)|^2\,d\xi\,dt\lesssim R^{N-s-k},
\end{equation}
where $\mu\in\mathcal M(\R^N)$ with $\mu(B(x,r)) \leq r^s$ for $x\in\R^N, r>0$, for some $0<s<n$.  By a similar argument as for Theorem \ref{obproj}, this leads to

\begin{thm}\label{obproj1}
Let $A\subset\R^N$ be a Borel set. 
\begin{itemize}
\item[(1)] If $\dim A \leq N-k$, then $\dim\pi_t(A)\geq\dim A - k(m-1)$ for almost all $t\in\R^k$.
\item[(2)]If $\dim A >N-k$, then $\mathcal L^m(\pi_t(A))>0$ for almost all $t\in\R^k$.
\end{itemize}
\end{thm}

Part (2) is again sharp. To see this, let $A$ consist of the points\\ $(x_0^1,\dots,x_0^m,x_1^1,\dots,x_1^m,\dots,x_k^1,\dots,x_k^m)\in\R^N$ for which $x^1_0\in C$, where $C$ has dimension 1 and measure zero, and $x^1_1=\dots=x^1_k=0$. Part (1) is sharp when $m=1$, but then $k=N-1$ and the standard Marstrand's projection theorem also applies. It also is sharp, for example, when $m=2$ for any $k$ with a similar example as in the case $k=1, m=2$. 

The study of restricted families of projections was started by E. J\"arvenp\"a\"a, M. J\"arvenp\"a\"a, Ledrappier and Leikas in \cite{JJLL}. This work was continued and generalized by the J\"arvenp\"a\"as and Keleti in \cite{JJK}, where they proved sharp inequalities for general smooth non-degenerate families of orthogonal projections onto $m$-planes in $\R^n$. Now the trivial examples such as $\{t(\cos\theta,\sin\theta,0): t\in\R\}, \theta\in [0,\pi]$, are also included, so the bounds are necessarily weaker than in the above special cases.  
Restricted  families appear quite naturally in Heisenberg groups, see  \cite{BDFMT}, \cite{BFMT} and  \cite{FH}. Another motivation for studying them comes from the work of E. J\"arvenp\"a\"a, M. J\"arvenp\"a\"a and Ledrappier and their co-workers on measures invariant under geodesic flows on manifolds, see  \cite{HJJL1} and  \cite{HJJL2}. 

\section{Plane sections and radial projections}

What can we say about the dimensions if we intersect a subset $A$ of $\Rn, \dim A>m,$ with $(n-m)$-dimensional planes? Using Proposition \ref{sectprop} we have for any $V \in G(n,n - m)$,
$$\dim (A \cap (V + x)) \leq \dim A - m\quad \text{for}\ \mathcal{H}^{m}\ \text{almost all}\ x\in V^{\perp},$$
and for any $x \in \Rn$ (see \cite{M1} or \cite{MO}),
$$\dim (A \cap (V + x)) \leq \dim A - m\quad \text{for}\ \gamma_{n,n - m}\ \text{almost all}\ V \in G(n,n - m).$$
The lower bounds are not as obvious, but we have the following result, originally proved by Marstrand in the plane in \cite{Ma1} and then in general dimensions in \cite{M1}:

\begin{thm}\label{sections}
Let $m < s \leq n$ and let $A\subset\R^n$ be $\mathcal H^s$ measurable with $0 < \mathcal{H}^{s}(A) < \infty$. Then 
\begin{itemize}
\item[(1)] For $\mathcal{H}^{s}$ almost all $x \in A$, $\dim (A \cap (V + x)) = s - m$ for $\gamma_{n,n - m}$ almost all $V \in G(n,n - m)$, 
\item[(2)] for $\gamma_{n,n - m}$ almost all $V \in G(n,n - m)$, 
$$\mathcal H^m(\{x\in V^{\perp}: \dim (A \cap (V + x)) = s - m\})>0.$$
\end{itemize}
\end{thm}

These statements are essentially equivalent. Clearly, this generalizes part (2) of Theorem \ref{projections}. Now we give exceptional set estimates related to both statements. The first of these is due to Orponen \cite{O1}:

\begin{thm}\label{O}
Let $m<s\leq n$ and let  $A\subset\Rn$ be $\mathcal H^s$ measurable with $0<\mathcal H^s(A)<\infty$. Then there is a Borel set 
$E \subset G(n,n-m)$ such that $\dim E \leq m(n-m) + m - s$ and for $V\in G(n,n-m)\setminus E$,
$$\mathcal H^m(\{x\in V^{\perp}: \dim (A \cap (V + x)) = s - m\})>0.$$
\end{thm}

The bound $m(n-m) + m - s = \dim G(n,n-m) +m - s$ is the same as in Theorem \ref{excproj}(2). Since it is sharp there, it also is sharp here.

The second estimate is due to Orponen and the author \cite{MO}:

\begin{thm}\label{main2}
Let $m<s\leq n$ and let  $A\subset\Rn$ be $\mathcal H^s$ measurable with $0<\mathcal H^s(A)<\infty$. Then there is a Borel set $B \subset \R^{n}$ such that $\dim B \leq m$ and for $x\in\Rn\setminus B$,
$$\gamma_{n,n-m}(\{V\in G(n,n-m):\dim A\cap(V+x)=s-m\}) >0.$$ 
\end{thm}

This probably is not sharp. I expect that the sharp bound for $\dim B$ in the case $m=n-1$ would again be $2(n-1) - s$, as for the orthogonal projections and as in Orponen's radial projection theorem \ref{radial} below. Moreover, one could hope for an exceptional set estimate including both cases, that is, estimate on the dimension of the exceptional pairs $(x,V)$. 

I give a sketch of the proof of Theorem \ref{main2} in the plane. Suppose that it is not true and that there is a set $B$ with $\dim B>1$ such that through the points of  $B$ almost all lines  meet $A$ in a set of dimension less than $s-1$. On the other hand, by Theorem \ref{sections} typical lines through the points of $A$ meet $A$ in a set of dimension $s-1$. By Fubini-type arguments and using Theorem \ref{main1} we can find such typical  lines meeting both $A$ and $B$ leading to a contradiction.

Here we investigated the dimensions of the intersections of our set with lines through a point. But if we only want to know whether these lines meet the set, we are studying radial projections. For these more can be said.  For $x\in\Rn$ define
$$\pi_x:\Rn\setminus\{x\}\to S^{n-1},\quad \pi_x(y)=\frac{y-x}{|y-x|}.$$
Then by the standard proofs the statements of Marstrand's projection theorem are valid for almost all $x\in\Rn$.
Orponen proved in \cite{O5} and \cite{O8} the following sharp estimate for the exceptional set of $x\in\Rn$.
\begin{thm}\label{radial}
Let  $A\subset\Rn$ be a Borel set with $\dim A>n-1$. Then there is a Borel set $B\subset\Rn$ with $\dim B\leq 2(n-1)-\dim A$ such that for every $x\in\Rn\setminus B,~ \mathcal H^{n-1}(\pi_x(A))>0$. Moreover, if $\mu\in\mathcal M(\Rn)$ and $I_s(\mu)<\infty$ for some $n-1<s<n$, then the push-forward of $\mu$ under $\pi_x$ is absolutely continuous with respect to $\mathcal H^{n-1}|\Sn$ for $x$ outside a set of Hausdorff dimension $2(n-1)-s$.
\end{thm}

Orponen proved in \cite{O8} also the following rather surprising result:

\begin{thm}\label{radial1}
Let  $A\subset\R^2$ be a Borel set with $\dim A>0$. Then the set 
$$\{x\in\R^2: \dim\pi_x(A) < \dim A/2\}$$
has Hausdorff dimension 0 or it is contained in a line.
\end{thm}

Obviously the second alternative is needed, since if $A$ is contained in a line, the above set is the same line.

\section{General intersections}
The following theorem was proved in \cite{M3}:

\begin{thm}\label{genint}
Let $s$ and $t$ be positive numbers with $s+t > n$ and $t>(n+1)/2$. Let $A$ and $B$ be Borel subsets of $\Rn$ with $\mathcal H^s(A)>0$ and $\mathcal H^t(B)>0$. Then 
for almost all $g\in O(n)$,
\begin{equation}\label{eq4}
\mathcal L^n(\{z\in\Rn: \dim A\cap (g(B)+z)\geq s+t-n\})>0.
\end{equation}
\end{thm}

The condition $t>(n+1)/2$ comes from some Fourier transform estimates. Probably it is not needed. 

This was preceded by the papers of Kahane \cite{K} and the author \cite{M2} in which it was shown that the above theorem is valid for any $s+t>n$ provided larger transformation groups are used. For example, it suffices to add also typical dilations $x\mapsto rx, r>0$. 

Here we really need the inequality $\dim A\cap (g(B)+z)\geq s+t-n$, the opposite inequality can fail very badly: for any $0\leq s\leq n$ there exists a Borel set $A\subset \Rn$ such that 
$\dim A\cap f(A)=s$ for all similarity maps $f$ of $\Rn$. This follows from \cite{F4}. The reverse inequality holds if  $\dim A\times B = \dim A + \dim B$, see \cite{M4}, Theorem 13.12. This latter condition is valid if, for example, one of the sets is Ahlfors-David regular, see \cite{M4}, 8.12. For such reverse inequalities no rotations $g$ are needed (or, equivalently, they hold for every $g$).

The following two exceptional set estimates were proven in \cite{M7}:

\begin{thm}\label{genint1}
Let $s$ and $t$ be positive numbers with $s+t > n+1$. Let $A$ and $B$ be Borel subsets of $\Rn$ with $\mathcal H^s(A)>0$ and $\mathcal H^t(B)>0$. Then 
there is a Borel set $E\subset O(n)$ such that 
$$\dim E\leq 2n-s-t+(n-1)(n-2)/2=n(n-1)/2-(s+t-(n+1))$$ 
and for  $g\in O(n)\setminus E$,
\begin{equation}
\mathcal L^n(\{z\in\Rn: \dim A\cap (g(B)+z)\geq s+t-n\})>0.
\end{equation}
\end{thm}

Notice that $n(n-1)/2$ is the dimension of $O(n)$. 
The condition $s+t > n+1$ is not needed in the case where one of the sets has small dimension and in this case we have a better upper bound for $\dim E$, although we then need a slight technical reformulation:

\begin{thm}\label{genint2}
Let $A$ and $B$ be Borel subsets of $\Rn$ with $\dim A=s, \dim B=t$ and suppose that $s\leq (n-1)/2$. If $0<u<s+t - n$, then there is a Borel set $E\subset O(n)$ with 
$$\dim E\leq n(n-1)/2-(s+t-n)$$ 
such that for  $g\in O(n)\setminus E$,
\begin{equation}\label{eq18}
\mathcal L^n(\{z\in\Rn: \dim A\cap (g(B)+z)\geq u\})>0.
\end{equation}
\end{thm}

The formulation in \cite{M7} is slightly weaker, but it easily implies the above. What helps here is the following sharp decay estimate for quadratic spherical averages for Fourier transforms of measures with finite energy:

$$\int_{|v|=1}|\widehat{\mu}(rv)|^2\,dv \leq C(n,s)I_s(\mu)r^{-s},\quad r>0,\ 0<s\leq (n-1)/2.$$

Such an estimate is false for $s>(n-1)/2$. There are sharp estimates in the plane by Wolff \cite{W3},  and good, but perhaps not sharp, estimates in higher dimensions by Erdo\u gan \cite{E}. More precisely, for $s\geq n/2$ and $\epsilon>0$,
\begin{equation}\label{wolff-e}
\int_{|v|=1}|\widehat{\mu}(rv)|^2\,dv \leq C(n,s)I_s(\mu)r^{\epsilon-(n+2s-2)/4},\quad r>0.
\end{equation}
This is very useful for distance sets, as discussed below, but gives very little for the intersections.
The proof uses restriction and Kakeya methods and results. In particular, the case $n\geq 3$ relies on Tao's bilinear restriction theorem.
These are discussed in \cite{M6}.  

Let us speculate about the possible sharp estimates in the plane. In Theorem \ref{genint1} we have the upper bound $4-(s+t)$ and in Theorem \ref{genint2} we have $3-(s+t)$. Could the second estimate be valid whenever $s+t>2$? This would mean that the dimension is $0$ when $s+t>3$. Could the exceptional set even be countable then? I don't think so, but I don't have a counter-example. Anyway, it need not be empty whatever the dimensions are. That is, using only translations we cannot say much for general sets. The following example follows from \cite{M2}, or see \cite{Ke1} for having $A=B$: there are compact subsets $A$ and $B$ of $\Rn$ such that $\dim A = \dim B =n$ and $A\cap (B+z)$ contains at most one point for every $z\in\Rn$. 

A problem related both to projections and intersections is the distance set problem. For $A\subset\Rn$ define the distance set

$$D(A) = \{|x - y|: x,y \in A\}\subset [0,\infty). $$
The following Falconer's conjecture seems plausible:
\begin{conjecture}\label{Falconer-conj}
If $n\geq 2$ and $A\subset\Rn$ is a Borel set with $\dim A>n/2$, then $\mathcal L^1(D(A))>0$, or even $\operatorname{Int}(D(A))\neq\emptyset$.
\end{conjecture}
Falconer \cite{F5} proved in 1985 that $\dim A>(n+1)/2$ implies $\mathcal L^1(D(A))>0$, and we also have then $\operatorname{Int}(D(A))\neq\emptyset$ by Sj\"olin and myself \cite{MS}. Here appears the same bound $(n+1)/2$ as for the intersections, and for the same reason. In both cases for a measure $\mu$ with finite $s$-energy estimates for the measures of the narrow annuli, $\mu(\{y:r<|x-y|<r+\delta\})$, for $\mu$ typical centers $x$ are useful. They are rather easily derived with the help of the Fourier transform if $s\geq (n+1)/2$.

The best known result is due to Wolff \cite{W3} for $n=2$ and to Erdo\u gan \cite{E} for $n\geq 3$:
\begin{thm}
If $n\geq 2$ and $A\subset\Rn$ is a Borel set with $\dim A>n/2+1/3$, then $\mathcal L^1(D(A))>0$.
\end{thm}
The proof is based on the estimate \eqref{wolff-e}.

The relation to projections appears when we look at the pinned distance sets:
$$D_x(A) = \{|x - y|: y \in A\}\subset [0,\infty),\quad x\in\Rn. $$
Peres and Schlag proved in \cite{PS} that these too have positive Lebesgue measure for many $x$ provided $\dim A>(n+1)/2$. We can think of $D_x(A)$ as the image of $A$ under the projection-type mapping $y\mapsto |x-y|$.

Various partial results on distance sets have recently been proved, among others, by Iosevich and Liu \cite{IL1}, \cite{IL2}, Luc\'a and Rogers \cite{LR}, Orponen \cite{O4} and Shmerkin \cite{S2}, \cite{S3}.

\section{Besicovitch and Furstenberg sets}

We say that a set in $\R^n, n\geq2,$ is a \emph{Besicovitch set}, or a Kakeya set, if it has zero Lebesgue measure and it contains a line segment of unit length in every direction. This means that for every $e\in S^{n-1}$ there is $b\in\Rn$ such that $\{te+b:0<t<1\}\subset B$. It is not obvious that Besicovitch sets exist but they do in every $\R^n, n\geq 2$:

\begin{thm}\label{bes} 
For any $n\geq 2$ there exists a Borel set~$B\subset\mathbb{R}^n$ such that $\mathcal{L}^n(B)=0$ and $B$ contains a whole line in
every direction. Moreover, there exist compact Besicovitch sets in $\Rn$. 
\end{thm}

\begin{proof}
It is enough to prove this in the plane, then $B\times\R^{n-2}$ is fine in $\Rn$. 
We shall use projections and duality between points and lines. More precisely, parametrize the lines, except those parallel to the $y$-axis, by $(a,b)\in \R^2$:
$$l(a,b)=\{(x,a+bx): x\in\R\}.$$ 
Then if $C\subset\R^2$ is some parameter set and $B=\cup_{(a,b)\in C}l(a,b)$, one checks that
$$B\cap\{(t,y):y\in\R\} = \{t\}\times\pi_t(C)$$
where
$$\pi_t:\R^2\to\R^2,\quad \pi_t(a,b)=a+tb,$$
is essentially an orthogonal projection. Suppose that we can find $C$ such that $\pi(C)=[0,1]$, where $\pi(a,b)=b$, and $\mathcal L^1(\pi_t(C))=0$ for almost all $t$. Then $\mathcal L^2(B)=0$ by Fubini's theorem and taking the union of four rotated copies of $B$ gives the desired set. It is not trivial that such sets $C$ exist but they do. For example, a suitably rotated copy of the product of a standard Cantor set with dissection ratio $1/4$ with itself is such, cf., for example, \cite{M6}, Chapter 10. Restricting $x$ above to a compact subinterval of $\R$ yields a compact Besicovitch set.
\end{proof}

The idea to construct Besicovitch sets using duality between lines and points is due to  Besicovitch from 1964 in \cite{B}, although he gave a geometric construction already in 1919. It was further developed by Falconer in \cite{F3}. We shall see more of this below.

\begin{conjecture}[Kakeya conjecture]
All Besicovitch sets in $\Rn$ have Hausdorff dimension $n$.
\end{conjecture}

The Kakeya conjecture is open for $n\geq 3$. I shall discuss partial results later, but let us first see how it follows in the plane and how it is related to projection theorems. The following theorem was proved by Davies in \cite{D}:

\begin{thm}\label{davies}
For every Besicovitch set $B\subset\R^n$, $\dim B \geq 2$. In particular, the Kakeya conjecture is true in the plane.
\end{thm}

The proof of this is, up to some technicalities, reversing the above argument for the proof of Theorem \ref{bes} and using Marstrand's projection Theorem \ref{projections}(1), see the proof of Theorem \ref{proj-bes} below. But let us now look more generally relations between projection theorems and lower bounds for the Hausdorff dimension of Besicovitch sets.

We can parametrize the lines in $\Rn$, except those orthogonal to the $x_1$-axis, by $(a,b)\in \R^{n-1}\times\R^{n-1}$:
$$l(a,b)=\{(x,a+bx): x\in\R\}.$$ 
Then again if $C\subset\R^{2(n-1)}$ is parameter set and $B=\cup_{(a,b)\in C}l(a,b)$ we have for $t\in\R$,
$$B\cap\{(t,y):y\in\R^{n-1}\} = \{t\}\times\pi_t(C)$$
where
$$\pi_t:\R^{2(n-1)}\to\R^{n-1},\quad \pi_t(a,b)=a+tb,\quad t\in\R.$$
These are projections of Section 4 with $k=1, m=n-1$. 
Suppose now that $\pi(C)=[0,1]^{n-1}$, where $\pi(a,b)=b$. Then in particular, $\dim C\geq n-1$. The projection theorem we would need to solve the Kakeya conjecture should tell us that $\dim\pi_t(C)=n-1$ for almost all $t\in \R$. Then we could conclude by Proposition \ref{sectprop} that $\dim B=n$. In the plane such projection theorem is true; it is just Marstrand's projection theorem. However, in higher dimensions we don't know of any such projection theorem since we now only have a one-dimensional family of projections. Notice that the space of all orthogonal projections from $\R^{2(n-1)}$ onto $(n-1)$-planes is $(n-1)^2$-dimensional. More precisely, we can state

\begin{thm}\label{proj-bes} Let $0<s\leq n-1$ and $\pi(x,y)=y$ for $x,y\in\R^{n-1}$. Suppose that the following projection theorem holds: For every Borel set $C\subset\R^{2(n-1)}$ with $\mathcal H^{n-1}(\pi(C))>0$ we have $\dim\pi_t(C)\geq s$ for almost all $t\in \R$. Then for every Besicovitch set $B\subset\Rn$, we have $\dim B\geq s+1$. In particular, if this projection theorem holds for $s=n-1$, the Kakeya conjecture is true.
\end{thm}

\begin{proof} We may assume that $B$ is a $G_{\delta}$-set, since any set in $\R^{n-1}$ is contained in a $G_{\delta}$-set with the same dimension.  
For $a\in\R^{n-1},b\in[0,1]^{n-1}$ and $q\in\Q$ denote by $I(a,b,q)$ the line segment $\{(q+t,a+bt):0\leq t\leq 1/2\}$ of length less than  $1$. Let $C_q$ be the set of~$(a,b)$ such that $I(a,b,q)\subset B$. Then each $C_q$ is a $G_{\delta}$-set, because for any open set $G$ the set of~$(a,b)$ such that $I(a,b,q)\subset G$ is open. 
Since for every $b\in[0,1]^{n-1}$ some $I(a,b,q)\subset B$, we have $\pi(\cup_{q\in\Q}C_q)=[0,1]^{n-1}$, so there is $q\in\Q$ for which 
$\mathcal{H}^{n-1}(\pi(C_q))>0$. Then by our assumption, for almost all $t\in\mathbb{R}$, $\dim\pi_t(C_q)\geq s$. 
We now have for $0\leq t\leq 1/2$,
\begin{equation*}
\{q+t\}\times\pi_t(C_q)=\{(q+t,a+bt):(a,b)\in C_q\}\subset B\cap \{(x,y):x=q+t\}.
\end{equation*}
Hence for a positive measure set of $t$, vertical $t$-sections of~$B$ have dimension at least $s$. By 
 Proposition \ref{sectprop} we obtain that $\dim B\geq s+1$.
\end{proof}

Let us try to apply Oberlin's projection theorem \ref{obproj1} together with Theorem \ref{proj-bes}. We have to apply it in $\R^{2(n-1)}$ with $k=1, m=n-1$. We have 
$\dim C \geq n-1$, so we get $\dim\pi_t(C)\geq n-1-(n-2)=1$, thus yielding the lower bound $2$ for the Hausdorff dimension of Besicovitch sets. But this also follows by Theorem \ref{davies}, and by other methods, see \cite{M6}.
Unfortunately no known method seems to give any better projection theorem for the family $\pi_t$. From $\mathcal H^{n-1}(C)>0$ we could only hope to get $\dim\pi_t(C)\geq (n-1)/2$, at least when $n$ is odd. To see this let $p=(n-1)/2$ and  $C=\{(a,b)\in\R^{n-1}\times\R^{n-1}:a_1=\dots=a_p=b_1=\dots =b_p=0\}$. Then $\mathcal H^{n-1}(C)=\infty$ and $\pi_t(C)=\{x\in\R^{n-1}:x_1=\dots=x_p=0\}$, so $\dim\pi_t(C)=(n-1)/2$. Even if this estimate were true it would only give the lower bound $(n+1)/2$ for the dimension of Besicovitch sets. This has been known since the 1980s by different methods, see \cite{M6}, Section 23.4. The only hope for better estimates via projections would seem to be that instead of only using the information $\mathcal H^{n-1}(C)>0$ we should use that $C$ has positive measure projection on the second factor of $\R^{n-1}\times\R^{n-1}$ Often having one big projection does not help much. However F\"assler and Orponen were able to make use of that in \cite{FO}, and since we are dealing with a very special family of mappings maybe it could help here too. Moreover, in the known cases the generic dimension of the projections agrees with the largest one. 

Yu proved in \cite{Y} that the Kakeya conjecture is equivalent to the following: for any Besicovitch set $B\subset\Rn$ and for any $0<m<n$,~ $\dim P_V(B)$ is constant for $V\in G(n,m)$. The idea is simple but clever: lift your Besicovitch set $B$ from $\Rn$ to $\R^{2n-1}$ in the way it projects back to $\Rn$ as $B$ and it projects to some $n$-dimensional subspace of $\R^{2n-1}$ as a Besicovitch set where all the defining lines go through the origin. Then this latter projection has positive $n$-dimensional measure.

So the Kakeya conjecture is true in the plane and open in higher dimensions. The following results give the best known lower bounds for the Hausdorff dimension of Besicovitch sets.

Wolff, based on some earlier work of Bourgain, proved in  \cite{W1}

\begin{thm}\label{wolf}
The Hausdorff dimension of every Besicovitch set in $\Rn$ is at
least $(n+2)/2$.
\end{thm}

Wolff's method is geometric. He proved the following Kakeya maximal function inequality which yields Theorem \ref{wolf} rather easily. 

\begin{equation}\label{hairbrush}
\|\mathcal K_{\delta}f\|_{L^{\frac{n+2}{2}}(\Sn)}\leq C(n,\epsilon)\delta^{\frac{2-n}{2+n}-\epsilon}\|f\|_{L^{\frac{n+2}{2}}(\Rn)}
\end{equation}
for all $\delta, \epsilon>0$. Here 
$$\mathcal K_{\delta}f(e)=\sup_{a\in\R^n}\frac{1}{\mathcal L^n(T_e^{\delta}(a))}\int_{T_e^{\delta}(a)}|f|\,d\mathcal L^n,$$
where $T_e^{\delta}(a)$ is the tube with center $a\in\Rn$, direction $e\in \Sn$, width $\delta$ and length $1$.

Wolff's estimate $\dim B \geq 3$ is still the best known in $\R^4$. 

Bourgain introduced in \cite{B2} a combinatorial method, further developed by Katz and Tao \cite{KT1} in \cite{KT2}, which led to the following:

\begin{thm}
For any Besicovitch set $B$ in $\Rn$, $\dim B\geq (2-\sqrt{2})(n-4)+3$. 
\end{thm}

This is the best known lower bound for $n\geq 5$ . Quite recently Katz and Zahl \cite{KZ} were able to establish an epsilon improvement on Wolff's bound $5/2$ in $\R^3$. Thus in $\R^3$ the best known estimate is

\begin{thm}
For any Besicovitch set $B$ in $\R^3$, $\dim B\geq 5/2+\epsilon$ where $\epsilon$ is a small constant. 
\end{thm}

The arguments of Katz and Zahl are very involved and complicated combining many earlier ideas. A new feature are the algebraic polynomial methods, first used by Dvir in \cite{Dv} to solve the Kakeya conjecture in finite fields. The polynomial methods have recently been used in many connections, an excellent treatise on these is Guth's book \cite{G}. Orponen applied them to projections in \cite{O9}. 

Let us now look at some relations between unions of lines and line segments. Keleti made the following conjecture in \cite{Ke2}:

\begin{conjecture}\label{keleti1} 
If $A$ is the union of a family of line segments in $\R^n$ and $B$ is the union of the corresponding lines, then $\dim A=\dim B$.
\end{conjecture} 

This is true in the plane, as proved by Keleti:

\begin{thm}\label{keleti} 
Conjecture \ref{keleti1} is true in $\R^2$.
\end{thm} 

If Keleti's conjecture is true in $\Rn$ for all $n\geq 3$,  it gives a lot of new information on the dimension of Besicovitch sets:

\begin{thm}[Keleti \cite{Ke2}]
(1) If Conjecture \ref{keleti1} is true for some $n$, then, for this $n$, every Besicovitch set in $\Rn$ has Hausdorff dimension at least $n-1$.

(2) If Conjecture \ref{keleti1} is true for all  $n$, then every Besicovitch set in $\Rn$ has upper Minkowski and packing dimension  $n$ for all  $n$.
\end{thm} 

\begin{proof}
Let $F$ be the projective transformation
$$F(\tilde x,x_n)=\frac{1}{x_n}(\tilde x,1),\quad (\tilde x,x_n)\in\R^{n-1}\times\R, x_n\neq 0.$$
Then for $e\in\Sn, e_n\neq 0, a\in \R^{n-1}, F$ maps the punctured line $l(e,a)=\{te+(a,0):t\neq 0\}$ onto the punctured line 
$\{u(a,1)+\frac{1}{e_n}(\tilde e,0): u\neq 0\}$. If $B$ contains a line segment on $l(e,a_e), e\in\Sn$, then $F(B)$ contains a line segment on $F(l(e,a_e)), e\in\Sn$. The line extensions of these latter punctured lines cover $\{x:x_n=0\}$ so $\dim F(B)\geq n-1$ provided Conjecture \ref{keleti1} is true. Clearly, $F$ does not change Hausdorff dimension, whence $\dim B\geq n-1$ and (1) holds.

(2) follows by the well-known trick of taking products and by the product inequalities \eqref{productm} and \eqref{productp}. Suppose that Conjecture \ref{keleti1} is true for all  $n$ and there exists a Besicovitch set $B$ in $\Rn$ with $\dim_PB<n$ for some $n$. Then $B^k\subset\R^{kn}$ would satisfy by \eqref{productp}
$$\dim B^k \leq \dim_PB^k \leq k\dim_PB<kn-1$$
for large $k$. This contradicts part (1) since $B^k$ is a Besicovitch set in $\R^{kn}$.
\end{proof}

Using Theorem \ref{marthm} Falconer and I proved in \cite{FM} that in Theorem \ref{keleti} line segments can be replaced by sets of positive one-dimensional measure. Later H\'era, Keleti and  M\'ath\'e in \cite{HKM} proved that sets of dimension one are enough. These methods and results extend to subsets of hyperplanes in $\Rn$, but they do not extend to lower dimensional planes. In particular they do not apply to Besicovitch sets in higher dimensions. 

More generally, we can investigate the following question: suppose $E$ is a Borel family of affine $k$-planes in $\Rn$. How does the Hausdorff dimension of $E$ (with respect to a natural metric) affect the Lebesgue measure and the Hausdorff dimension of the union $L(E)$ of  these planes, or of $B\cap L(E)$ if we know that $B$ intersects every $V\in E$ in a positive measure or in dimension $u$? Oberlin used in \cite{Ob3} the projection theorems of Section 4 to prove that  $\dim E > (k+1)(n-k) - k$ implies $\mathcal L^n(L(E))>0$, and this is sharp. He also proved some lower bounds for the dimension, which are sharp when $k=n-1$ and $0<s\leq 1$, then the lower bound is $n-1+s$, but they probably are not always sharp. 

H\'era, Keleti and  M\'ath\'e studied in \cite{HKM} questions of the above type and proved many interesting generalizations of the above results. For example they proved the following

\begin{thm}
Let $1\leq k<n$ be integers and $0\leq s\leq 1$. If $E$ is a non-empty family of affine $k$-planes in $\Rn$ with $\dim E=s$ and $B\subset L(E)$ such that $\dim B\cap V=k$ for every $V\in E$, then
$$\dim B = \dim L(E) = s+k.$$
\end{thm}

Again, the right hand equality can fail if $s>1$; consider for example more than 1-dimensional families of lines in a plane. But the left hand inequality might hold always. However it is unknown for $s>1$.

{\bf Furstenberg sets} are kind of fractal versions of Besicovitch sets. We consider them only in the plane. For Besicovitch sets we had a line segment in each direction. We would still have dimension 2 if we would replace line segments with sets of dimension 1. But things get much more difficult if we replace them with lower dimensional sets. We say that $F\subset\R^2$ is a Furstenberg $s$-set, $0<s\leq 1$, if for every $e\in S^1$ there is a line $L_e$ in direction $e$ such that $\dim F\cap L_e \geq s$. What can be said about the dimension of $F$? Wolff \cite{W4}, Section 11.1, showed that 
\begin{equation}\label{woff-fu}
\dim F\geq \max\{2s,s+1/2\}
\end{equation} 
and that there is such an $F$ with $\dim F = 3s/2+1/2$. He conjectured that $\dim F \geq 3s/2+1/2$ would hold for all Furstenberg $s$-sets.  
When $s=1/2$ Bourgain   \cite{B3} improved the lower bound 1 to $\dim F\geq 1+c$ for some absolute constant $c>0$. 

Oberlin \cite{Ob4} observed a connection to projections, and in particular to dimension estimates for exceptional projections and Conjecture \eqref{Ob}. In this way he improved Wolff's estimates for some particular Furstenberg sets.  Let us see how this goes.  

Let $E\subset\R$ be a Borel set with $\dim E = s$ and $C\subset\R^2$ a parameter set for our lines such that $\pi(C)=\R, \pi(x,y)=y$, whence $\dim  C\geq 1$. Set
$$F=\{(x,a+bx):x\in E, (a,b)\in C\}.$$
Then $F$ is (essentially) a Furstenberg $s$-set. As before for $t\in E$,
$$F\cap\{(t,y):y\in\R\} = \{t\}\times\pi_t(C)$$
where
$$\pi_t:\R^2\to\R^2,\quad \pi_t(a,b)=a+tb.$$
Let $0<u<(s+1)/2$. If Conjecture \eqref{Ob} holds, we obtain
$$\dim\{t:\pi_t(C)<u\} \leq 2u - 1 < s = \dim E.$$
Hence there is $E_1\subset E$ such that $\dim E_1=s$ and $\dim\pi_t(C)\geq u$ for $t\in E_1$. It follows by Proposition \ref{sectprop}  that $\dim F \geq s+u$. Letting $u\to (s+1)/2$, we get $\dim F \geq 3s/2+1/2$.

Thus the projection conjecture \eqref{Ob} implies Wolff's conjecture for these special Furstenberg sets. Even for these no better dimension estimate is known than \eqref{woff-fu}. Oberlin proved a better estimate, but weaker than the conjectured one, in the case where $C=C_1\times C_1$ and $C_1\subset\R$ is the standard symmetric Cantor set of dimension $1/2$. He did this by improving Kaufman's estimate $\dim\{t:\pi_t(C)<u\} \leq u$ in this case. 

Orponen has proved (unpublished) that if we have the lower bound $t+(2-t)s$ for some $t\in[0,1/2]$ for the  Hausdorff dimension of all Furstenberg $s$-sets $F\subset\R^2$, then
$$\dim\{e\in S^1: \dim_MP_e(F)\leq u\}\leq\max\left\{\frac{u-t}{1-t},0\right\}\quad \text{for}\ 0\leq u\leq 1.$$

Orponen improved in \cite{O7} Wolff's bound for the packing dimension:

\begin{thm}
For $1/2<s<1$ there exists a positive constant $\epsilon(s)$ such that for any Furstenberg $s$-set $F\subset\R^2$ we have $\dim_PF>2s+\epsilon(s)$. 
\end{thm}

Recall Orponen's packing dimension estimate for projections \eqref{O2}. Proofs for these two results are rather similar, and based on combinatorial arguments.

This dimension problem is related to Furstenberg's question on sets invariant under $x\mapsto px(mod 1), x\in\R, p\in\Z$. This problem was recently solved, independently and by different methods, by Shmerkin \cite{S1} and by Wu \cite{Wu}.  

Other recent results on Furstenberg sets are due to Molter and Rela \cite{MR1}, \cite{MR3} and \cite{MR2},   and Venieri \cite{V}. Rela has a survey in \cite{R}. 

One reason for the great interest in Besicovitch sets and Kakeya conjecture is that the restriction conjecture
 
$$\|\widehat f\|_{L^q(\R^n)} \leq C(n,q)\|f\|_{L^{\infty}(S^{n-1})}\quad \text{for}\ q>2n/(n-1),$$

implies the Kakeya conjecture. For more on this, see for example \cite{W4} and \cite{M6}.

\section{$(n,k)$ Besicovitch sets}

We obtain other interesting Besicovitch set problems by replacing lines with higher dimensional planes.

\begin{df}\label{be1}
A  set $B\subset\Rn$ is said to be an $(n,k)$ \emph{Besicovitch set}  if $\mathcal L^n(B)=0$ and there is a non-empty open set $G\subset G(n,m)$ such that for every $V\in G$ there is $a\in\Rn$ such that $B(a,1)\cap(V+a)\subset B.$

We say that a set $B\subset\Rn$ is a full $(n,k)$ \emph{Besicovitch set}  if $\mathcal L^n(B)=0$ and there is a non-empty open set $G\subset G(n,m)$ such that for every $V\in G$ there is $a\in\Rn$ such that $V+a\subset B.$
\end{df}

We have used the open set $G$ in this definition for later convenience. Our main interest is for what pairs $(n,k)$ such sets exist and for this it is equivalent to use $G=G(n,k)$. 

Extending earlier results of Marstrand \cite{Ma3} ($n=3, k=2$), Falconer \cite{F1} ($k>n/2$) and Bourgain \cite{B1} ($2^{k-1}+k\geq n$)  R. Oberlin \cite{Ob} proved  that there exist no $(n,k)$ Besicovitch sets if $(1+\sqrt{2})^{k-1}+k > n$. For other values of $k\geq 2$ their existence is unknown. Let us now see how this relates to projections.

Mimicking the arguments from the previous section we only consider affine $k$-planes in $\Rn$ which are graphs over $\R^k$ identified with the coordinate plane $x_{k+1}=\dots=x_n=0$. They can be parametrized as
$$L(l,c)=\{(x,lx+c):x\in\R^k\},\quad l\in L(\R^k,\R^{n-k}), c\in\R^{n-k},$$ 
where $L(\R^k,\R^{n-k})$ is the space of linear maps from $\R^k$ into $\R^{n-k}$, identified with $\R^{k(n-k)}$. Let $\pi:\R^{k(n-k)}\times\R^k\to \R^{k(n-k)}$ with $\pi(l,c)=l$. Suppose we could find a Borel set $C\subset \R^{k(n-k)}\times \R^{n-k}$ for which the interior of $\pi(C)$ is non-empty and $\mathcal L^n(B)=0$ where
$$B=\bigcup_{(l,c)\in C}L(l,c).$$
Then $B$ would be a full $(n,k)$ Besicovitch set. Define 
$$\pi_t:L(\R^k,\R^{n-k})\times\R^{n-k}\to \R^{n-k},\quad \pi_t(l,c)=lt+c, (l,c)\in L(\R^k,\R^{n-k})	\times\R^{n-k}, t\in \R^k.$$
For $t\in\R^k$ we now have
$$B\cap\{(x,y)\in\R^k\times\R^{n-k}: x=t\} = \{t\}\times\pi_t(C).$$
So by Fubini's theorem $\mathcal L^n(B)>0$ if and only if $\mathcal L^{n-k}(\pi_t(C))>0$ for $t$ in a set of positive $k$-dimensional Lebesgue measure.

Hence the question for which values of $n$ and $k$ the projection properties (P1) and (P2) below are valid is very close to the question of the existence of $(n,k)$ Besicovitch sets:\\

(P1) If $C\subset \R^{k(n-k)}\times\R^{n-k}$ is a Borel set for which the interior of 
$\pi(C)$ is non-empty, then $\mathcal L^{n-k}(\pi_t(C))>0$ for positively many $t\in \R^k$.\\

(P2) If $C\subset \R^{k(n-k)}\times\R^{n-k}$ is a Borel set with 
$\mathcal L^{k(n-k)}(\pi(C))>0$, then $\mathcal L^{n-k}(\pi_t(C))>0$ for almost all $t\in \R^k$.\\

(P3) If $C\subset \R^{k(n-k)}\times\R^{n-k}$ is a Borel set with 
$\mathcal H^{k(n-k)}(C)>0$, then $\mathcal L^{n-k}(\pi_t(C))>0$ for almost all $t\in \R^k$.\\

Clearly, (P3) implies (P2) implies (P1). Probably (P1) and (P2) are equivalent but it may be difficult to show this without really verifying their validity. Notice that (P3) is almost the same as statement (2) in Oberlin's theorem \ref{obproj1} in the case $m=n-k, N=(k+1)(n-k)$. We shall come back to that, and we shall see that (P1) does not always imply (P3).

If $k=n-1$, then the $\pi_t$ form an $(n-1)$-dimensional family of linear maps $\Rn\to\R$, which is essentially the same as the full family of orthogonal projections. Thus these statements are true by standard Marstrand's projection theorem and we regain by Proposition \ref{Pprop} below the nonexistence of $(n,n-1)$ Besicovitch sets. This was proved by Marstrand by a simple geometric method for $n=3$ and that proof easily generalizes. For other pairs $(n,k)$ the validity of (P1) and (P2) does not seem to have an obvious answer. But we can easily state some connections. 

\begin{pr}\label{Pprop}
\begin{itemize}
\item[(1)]  Full $(n,k)$ Besicovitch sets do not exist if and only if (P1) holds.
\item[(2)]  $(n,k)$ Besicovitch sets do not exist  if  (P2) holds.
\end{itemize}
\end{pr}

So if we would know that (P1) and (P2) are equivalent, we would know that the existence of full $(n,k)$ Besicovitch sets and of  $(n,k)$ Besicovitch sets is equivalent. 

\begin{proof}
Part (1) was already stated above. 

(2) can be proven with an easy modification of the argument that we gave for Theorem \ref{proj-bes}. Let $B\subset\Rn$ be a  $G_{\delta}$-set which contains a unit $k$-ball in every direction. We need to show that $\mathcal L^n(B)>0$. 
For $q\in \Q^k$, let $C_q$ be the set of $(l,c)$ such that $l$ belongs to the closed unit ball $B_L$ of $L(\R^k,\R^{n-k})$ and $(q+t,lt+c)\in B$ for $t\in B(0,1/2)\}$. Then  $|lt|\leq |t|$ for $t\in\R^k$. Again each $C_q$ is a $G_{\delta}$-set and $\pi(\cup_{q\in\Q^k}C_q)= B_L$, so there is $q\in\Q^k$ for which 
$\mathcal{H}^{k(n-k)}(\pi(C_q))>0$. Thus by (P2) $\mathcal L^{n-k}(\pi_t(C))>0$ for almost all $t\in \R^k$.
Since for $t\in B(0,1/2)$,
\begin{equation*}
\{q+t\}\times\pi_t(C_q)=\{(q+t,lt+c):(l,c)\in C_q\}\subset B\cap \{(x,y):x=q+t\}.
\end{equation*}
we conclude that $\mathcal L^n(B)>0$.

\end{proof}

Let us go back to the statement (2) in Oberlin's theorem \ref{obproj1} in the case $m=n-k$ and $N=(k+1)(n-k)$. If $C$ is as in (P2), then $\dim C\geq k(n-k)$. If $k(n-k) > (k+1)(n-k) - k$, that is $k>n/2$, then by Theorem \ref{obproj1} (P2) holds and we obtain by Proposition \ref{Pprop} that $(n,k)$ Besicovitch sets do not exist. This was proved by Falconer in \cite{F1} with a different Fourier-analytic method. As mentioned after Theorem \ref{obproj1}, (P3) fails if $k(n-k) < (k+1)(n-k) - k$.   Suppose now that $(1+\sqrt{2})^{k-1}+k\geq n$. Then by the above mentioned results of Bourgain and Oberlin and by Proposition \ref{Pprop}(1) (P1) holds. In particular, we obtain in a rather indirect way a projection theorem from the results of Bourgain and Oberlin. Perhaps their methods could be used more directly to prove also other interesting projection theorems. We also see now that for pairs $(n,k)$ for which both  $k< n/2$ and $(1+\sqrt{2})^{k-1}+k\geq n$, (P3) fails but (P1) holds. It would be interesting to see why this is so just using arguments with projections.

\vspace{1cm}
\begin{footnotesize}
{\sc Department of Mathematics and Statistics,
P.O. Box 68,  FI-00014 University of Helsinki, Finland,}\\
\emph{E-mail address:} 
\verb"pertti.mattila@helsinki.fi"

\end{footnotesize}

\end{document}